\documentclass[english]{amsart}

\usepackage{latexsym}

\usepackage{pstricks,pst-node,pst-tree}
\usepackage{amsmath}
\usepackage{rotating}
\usepackage{amsfonts}
\usepackage{amsthm}
\usepackage{amssymb}
\usepackage{verbatim}
\usepackage{psfrag}
\usepackage{graphicx}

\newtheorem{lemm}{Lemma}[section]
\newtheorem{prop}[lemm]{Proposition}

\newtheorem{coro}[lemm]{Corollary}
\newtheorem{rema}[lemm]{Remark}
\newtheorem{theo}{Theorem}

  \usepackage[all]{xy}
\xyoption{matrix}

\newcommand{\im}{\mathbf{i}}

\newcommand{\Z}[1]{\mathbb{Z}/#1\mathbb{Z}}

\newcommand{\C}{\mathbb{C}}
\newcommand{\Bir}{\mathrm{Bir}}

\newcommand{\Aut}{\mathrm{Aut}}

\newcommand{\Sym}{\mathrm{Sym}}

\newcommand{\Pic}[1]{\mathrm{Pic}(#1)}
\newcommand{\A}{\mathbb{A}}

\newcommand{\p}{\mathbb{P}}
\newcommand{\z}{\mathbb{Z}}
\newcommand{\SL}{\mathrm{SL}}

\title{Non-rationality of some fibrations associated to Klein surfaces}
\author{J\'er\'emy Blanc}
\begin{document}
\begin{abstract}
We study the polynomial fibration induced by the equation of the Klein surfaces obtained as quotient of finite linear groups of automorphisms of the plane; this surfaces are of type A, D, E, corresponding to their singularities.

The generic fibre of the polynomial fibration is a surface defined over the function field of the line. We  proved that it is not rational in cases D, E, although it is obviously rational in the case A.

The group of automorphisms of the Klein surfaces is also described, and is linear and of finite dimension in cases D, E; this result being  obviously false  in case A.

\end{abstract}
\thanks{The author aknowledge support from the Swiss National Science Foundation grant no PP00P2\_128422/1.}
\maketitle
\section{Introduction}
Taking any finite subgroup $G\subset \SL(2,\C)$ acting on $\C^2$, the quotient is a surface in $\C^3$ given by one of the following polynomials:
$$\begin{array}{rcll}
a_n(x,y,z)&=&x^n-yz,& n\ge 2\\
 d_n(x,y,z)&=&x^{n-1}+xy^2+z^2, & n\ge 4\\
e_6(x,y,z)&=&x^4+y^3+z^2,\\
 e_7(x,y,z)&=& x^3y+y^3+z^2,\\
 e_8(x,y,z)&=&x^5+y^3+z^2.
\end{array}
$$
The surfaces obtained here are often called \emph{Klein surfaces}, which have one singularity at the origin, of type $A,D,E$, and all such singularities are obtained by this process. There is a classical relation between the Dynkin diagram obtain from the resolution of singularities and the one given by the representation theory of $G$ (see for example~\cite{bib:Reid} and its references).\\

If $f\in \C[x,y,z]$ is any of the polynomials above, the map $F\colon \C^3\to \C$ given by $(x,y,z)\mapsto f(x,y,z)$ is a morphism of algebraic varieties. The fibre $f_t$ of $t\not=0$ is a smooth  hypersurface of $\C^3$,  whereas the fibre $f_0$ of $0$ is the quotient surface described before. One can then see $F$ as a deformation of the singularity.

In this article, we are interested in the question of the rationality of the generic fibre of $F$, asked to us by J. Alev. Even if the general fibres of $F$ is rational,  the generic fibre is not rational, except in the trivial case of $a_n$.
\begin{theo}\label{Thm:NonRat}
Let $f\in \C[x,y,z]$ be equal to one of the polynomials $a_n,d_n,e_6,e_7$ or $e_8$ above. The following are equivalent:
\begin{enumerate}
\item
The field extension $\C(x,y,z)/\C(f)$ is rational, i.e. there exists exist $g,h\in \C(x,y,z)$ such that $\C(x,y,z)=\C(f,g,h)$;
\item
The polynomial $f$ is equal to $a_n$ $($for some $n\ge 2)$.
\end{enumerate}
\end{theo}
\begin{rema}
The implication $(2)\Rightarrow (1)$ is easy, by taking $g=x$ and $h=y$ for instance. Although the case $A_n$ is trivial for the result of the theorem, we will apply our strategy of proof to it, in order to see its difference with the other cases from the geometric point of view.
\end{rema}

To prove Theorem \ref{Thm:NonRat}, we will provide a natural compactification of  the generic fibre of the the polynomial map $F\colon \C^3\to \C$, obtaining a minimal surface defined over $\C(t)$, whose root system is exactly of type $A_n$, $D_n$, $E_6$, $E_7$ or $E_8$. The surface, minimal over $\C(t)$, will be minimal over $\overline{\C(t)}$ only in the case $A_n$, and will be rational only in this case too. The technique only involves simple classical tools of birational geometry, with some little tricks. We will moreover find the minimal extension of $\C(t)$ which makes the extension rational:
\begin{theo}\label{Thm:NonRatDegre}
Let $f\in \C[x,y,z]$ be equal to one of the polynomials $d_n,e_6,e_7$ or $e_8$ above. 
There exists a number $a\in \mathbb{N}$ depending only on $f$ such that: 

For any field extension $K$ of $\C(f)$, the extension $K(x,y,z)/K$ is rational if and only if $K$ contains a $a$-th root of $f$.

Moreover, $a$ is equal to $12$, $18$, $30$ if $f=e_6$, $e_7$, $e_8$ and is the highest power of $2$ that divides $2(n-1)$ if $f=d_n$.
\end{theo}

\begin{rema}
$1)$ The dimensions of the minimal field extensions above are the Coxeter number of the Weyl groups $E_6,E_7,E_8$. For $D_n$, the Coxeter number is $2(n-1)$ and here we only find the highest power of $2$ that divides this number.

$2)$ The fact that the extraction of roots $t$ give rational surfaces can be checked directly by looking at the equations, and doing change of variables of the type $x=u^aX$, $y=u^bY$, $z=u^c Z$, $t=u^d$. The $d$ needed so that the equation get rid of $u,t$ is exactly the one given by the Theorem~$\ref{Thm:NonRatDegre}$.
\end{rema}

This result is the same as the result on the Gelfand-Kirillov conjecture for
simply laced simple Lie algebras established by A. Premet \cite{bib:Premet}: among the
ADE types, it is true if and only if we are in type A. This was the origin of
J. Alev's question and might be the starting point of an alternative proof of
Premet's result via the structure of transverse slices. Theorems~\ref{Thm:NonRat} and~\ref{Thm:NonRatDegre} could also yield a way to prove that the Gelfand-Kirillov conjecture does not hold for finite W-algebras of type $D$, $E$ (although it is true in type $A$, as proved in \cite{bib:FMO}).

Note that the question is also related to the question of rational variables of $\C^3$. A rational function $f\in \C(x,y,z)$ is a rational variable if there exists $g,h\in \C(x,y,z)$ such that $\C(f,g,h)=\C(x,y,z)$. The notion of variable is related to the understanding of the Cremona group $\Aut_\C(\C(x,y,z))=\Bir(\C^3)$. A necessary condition to be a variable is that the general fibres of $f\colon \C^3\dasharrow \C$ are rational. Theorem~\ref{Thm:NonRat} shows that this condition is not sufficient.
\\

Another question related to the Klein surfaces corresponds to describing the automorphisms group of these surfaces viewed as affine algebraic surfaces. 

The group of automorphisms of the surface of type $a_n$ has obviously infinite dimension, since it contains the group
$$\left.\left\{(x,y,z)\mapsto \left(x+yP(y),y,z+\frac{(x+yP(y))^n-x^n}{y}\right)\ \right|\ P\in \C[y]\right\}.$$
 It is in fact an amalgamated free product \cite[Theorem 5.4.5]{BD}.

In all other cases, we will prove that only linear automorphisms are possible (see Corollaries~\ref{Coro:KleinAuto} and \ref{Coro:KleinAutoDn}), and obtain the following result:
\begin{theo}
Let $f\in \C[x,y,z]$ be equal to one of the polynomials $a_n,d_n,e_6,e_7$ or $e_8$ above. The following are equivalent:
\begin{enumerate}
\item
Every automorphism of the Klein surface
$$\begin{array}{rcl}
S_{f}&=&\{(x,y,z)\in \C^3\ |\ f(x,y,z)=0 \}\\
\end{array}$$
extends to a linear automorphism of $\C^3$;
\item
The group of automorphisms of the Klein surface $S_{f}$ has finite dimension;
\item
The polynomial $f$ is  equal to $e_6,e_7,e_8$ or $d_n$ $($for some $n\ge 4)$.
\end{enumerate}

Each of the automorphisms of $S_{f}$ is moreover diagonal if $f=e_6,e_7,e_8$ or $d_n$ for $n\ge 5$.
Moreover, $\Aut(S_{e_6})\cong \C^{*}\times \{\pm 1\}$, $\Aut(S_{e_7})\cong \Aut(S_{e_8})\cong \C^{*}$, $\Aut(S_{d_n})\cong \C^{*}\times \Z{2}$ for $n\ge 5$ and $\Aut(S_{d_4})=\C^{*}\rtimes \Sym_3$.
\end{theo}We find once again a significative difference between the case $A$ and the cases $D,E$.\\

I thank Jacques Alev for asking me the above questions and for interesting discussions on the subject. Thanks also to Alexander Premet for some corrections on an earlier version of the paper.

\section{Rewriting of the problem of rationality}
Recall that a rational map between algebraic varieties is a map which is locally defined by quotient of polynomials. 
These "rational maps" are not  really maps because they are defined only on the open dense subset where the local quotients are determined
 A rational map is said to be a morphism if this domain is the whole variety, i.e.\ if the map is defined at every point. The rational map is birational if it admits an inverse, which is rational. 

We want to study the map $F\colon \C^3\to \C$ from the geometric and algebraic point of view. Algebraically, we will show that the extension $\C(x,y,z)/\C(f)$ is rational only in the case of $A_n$, whereas $\C(x,y,z)/\C$ and $\C(f)/\C$ are rational extenstions. Geometrically,  $F$ induces a morphism from a rational $3$-fold to a rational curve; we will show that the fibre of any (closed) point is a rational surface (i.e.  for any $t\in \C$ the complex surface $\{(x,y,z)\in \C^3 \ | \ f(x,y,z)=t\}$ is birational to $\p^2_\C$), whereas the generic fibre is rational only in the case of $A_n$  (i.e. the surface $\{(x,y,z)\in \C(t)^3 \ | \ f(x,y,z)=t\}$ defined over $\C(t)$ is birational to $\p^2_{\C(t)}$ only in the case  $A_n$).

Note that this situation is impossible in lower dimension. Indeed, for any rational map $\eta\colon \C^2\dasharrow \C$ having a general fibre which is rational, the generic fibre is rational, by Tsen Theorem. 
\section{Compactifications and rationality of them}

We will compactify the generic fibre of $F$ in order to use tools of projective geometry to show when it is rational or not.

Recall that any algebraic surface, over any perfect field $k$, is birational to a smooth projective surface $X$. One can moreover choose $X$ to be \emph{minimal}. A projective smooth surface $X$ defined over $k$ is said to be minimal if any birational morphism $X\to Y$, where $Y$ is another smooth algebraic surface, is an isomorphism. Let us say that a \emph{contractible curve} on $X$ is a curve $C\subset X$ defined over $k$, which is irreducible over $k$ and which decomposes over the algebraic closure $\bar{k}$ of $k$ into a set of disjoint $(-1)$-curves (curves isomorphic to $\p^1_{\bar{k}}$ and of self-intersection $-1$).  Since any birational morphism between smooth projective surfaces is a sequence of isomorphisms and contractions of contractible curves, the surface $X$ is minimal if and only if it does not contain contractible curves.

Recall the following classification of minimal geometrically rational surfaces (here $\Pic{S}$ denotes the Picard group of $S$, which corresponds to the group of divisors modulo linear equivalence, and $K_S$ is the canonical divisor):

\begin{prop}[\cite{bib:Man}, \cite{bib:Isk3}]\label{MinimalThenMori}
Let $S$ be a projective smooth surface defined over a perfect field $k$. If the surface $S$ is rational over $\overline{k}$ and minimal over~$k$, one of the following occurs: 

\begin{enumerate}
\item
$\Pic{S}\cong \z$ and $S$ is a del Pezzo surface $($which means that $-K_S$ is an ample divisor$)$;
\item
$\Pic{S}\cong \z^2$, and $S$ admits a conic bundle $\pi\colon S\to C$, where $C$ is a smooth rational curve of genus $0$ $($isomorphic to $\p^1$ over $\overline{k})$.  
\end{enumerate}
\end{prop}

We divide our study in two, corresponding to the two cases of Proposition~\ref{MinimalThenMori}, which will be investigated in $\S\ref{Sec:dPdf}$ and $\S\ref{Sec:CB}$ respectively.
\subsection{del Pezzo case}\label{Sec:dPdf}
Recall that a del Pezzo surface is a smooth projective surface $S$ with $-K_S$ ample. It has a degree, which is $(K_S)^2\in\{1,2,\dots,9\}$.

Over an algebraically closed field, $S$ is isomorphic to $\p^2$, $\p^1\times \p^1$ or to the blow-up of $1\le r \le 8$ points of $\p^2$ in general position (no $3$ collinear, no $6$ on the same conic, no $8$ on the same cubic being singular at one of the $8$ points); the degree is then $9-r$. For more details, see for instance \cite{bib:Dem} or \cite{bib:Be1}.

If $S$ is the blow-up of $p_1,\dots,p_r\in \p^2$, the Picard group of $S$ is generated by $e_0$, $e_1,\dots,e_r$, where $e_0$ is the pull-back of a line of $\p^2$ and $e_i$ is the exceptional curve contracted on $p_i$. The intersection form is of type $(1,-1,\dots,-1)$ (i.e. $(e_0)^2=1$, $(e_i)^2=-1$ for $i\ge 1$ and $e_i\cdot e_j=0$ if $i\not=j$). We associate to this a root system:
$$e_0-e_1-e_2-e_3, e_1-e_2,e_2-e_3,\dots,e_{r-1}-e_r$$
All this roots have self-intersection $-2$, and the intersection form between them is given by the following diagram, which is of type $E_6$, $E_7$, $E_8$ if $r=6,7,8$:

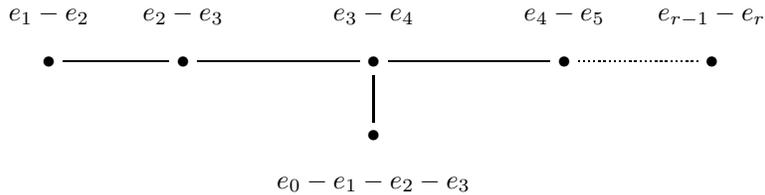
\[\xymatrix@R=0.2cm@C=0.5cm{
e_1-e_2& e_2-e_3 &e_3-e_4 &e_4-e_5 & e_{r-1}-e_r \\
\bullet \ar@{-}[r]&\bullet \ar@{-}[r]&\bullet\ar@{-}[r]&\bullet\ar@{.}[r]&\bullet\\
\\
& & \bullet\ar@{-}[uu]\\
& & e_0-e_1-e_2-e_3
}\]

Recall what are the classical canonical embeddings of del Pezzo surfaces of degree $\le 3$ (see for instance \cite{bib:Kol}, Theorem III.3.5).

If $S$ is a del Pezzo surface $S$ of degree $3$ over any field, the  anticanonical system $|-K_S|$ yields an isomorphism of $S$ with a smooth cubic in $\p^3$. Moreover, all smooth cubics are obtained by this way.

 If $S$ is a del Pezzo surface of degree $2$ over any field, the anticanonical system $|-K_S|$ gives a double covering of $\p^2$ ramified over a smooth quartic, and all smooth quartics are obtained by this way. We can then embedd $S$ into a weighted projective space $\p(1,1,1,2)$, obtaining an equation of degree $4$. Moreover, the linear system $|-mK_S|$ is given by the trace of the system of hypersurfaces of degree $m$ of $\p(1,1,1,2)$. 
 
 If $S$ is a del Pezzo surface of degree $1$ over any field, the anticanonical system $|-K_S|$ yields an elliptic fibration $S\to \p^1$ and $|-2K_S|$ gives a double covering of a quadric cone in $\p^3$. We can then embedd $S$ into a weighted projective space $\p(1,1,2,3)$, obtaining an equation of degree $6$. Moreover, the linear system $|-mK_S|$ is given by the trace of the system of hypersurfaces of degree $m$ of  $\p(1,1,2,3)$. 
 
 Note that $(-1)$-curves on a cubic surface correspond to the $27$ lines of the surfaces, which are maybe not all defined over the base-field. On a del Pezzo surface of degree $2$, the $56$ $(-1)$-curves corresponds to the $28$ bitangents of the quartic, which gives members of $|-K_S|$ decomposing into two curves. Note that here we call a bitangent a curve being tangent twice at the points \emph{or} intersecting the curve at only one point. On a del Pezzo surface of degree $1$, the $240$ $(-1)$-curves corresponds to the member of $|-2K_S|$ which decompose into two curves.\\

%
%

In cases $E_6$, $E_7$, $E_8$, we compactify the generic fibre of $F\colon \C^3 \to \C$ into a minimal del Pezzo surface of degree $3$, $2$, $1$. These are the following:

\medskip

{\bf Case $E_8$ -- } The compactification is 
$$S_8=\{(W:X:Y:Z)\in \p(1,1,2,3)_{\C(t)}\ |\ tW^6=X^5W+Y^3+Z^2\},$$
which is a del Pezzo surface of degree $1$. The morphism $S_8\to \p^1_{\C(t)}$ given by $$(W:X:Y:Z)\mapsto (W:X)$$ is an elliptic fibration, and the morphism $S_8\to \p(1,1,2)_{\C(t)}$ given by $$(W:X:Y:Z)\mapsto (W:X:Y)$$ is a double covering.

\medskip

{\bf Case $E_7$ -- } The compactification is 
$$S_7=\{(W:X:Y:Z)\in \p(1,1,1,2)_{\C(t)}\ |\ tW^4=X^3Y+Y^3W+Z^2\},$$
which is a del Pezzo surface of degree $2$. The morphism $S_7\to \p^2_{\C(t)}$ given by $$(W:X:Y:Z)\mapsto (W:X:Y)$$ is a double covering ramified over the smooth quartic $tW^4=X^3Y+Y^3W$.

\medskip

{\bf Case $E_6$ -- } The first simple compactification is 
$$S_6'=\{(W:X:Y:Z)\in \p(1,1,1,2)_{\C(t)}\ |\ tW^4=X^4+Y^3W+Z^2\},$$
which is a del Pezzo surface of degree $2$.  But the surface is not minimal, it contains two $(-1)$-curve defined over $\C(t)$, which are $W=0, Z=\pm \mathbf{i} X^2$. Note that the equation of $S_6'$ can be written as $(Z-\im X^2)(Z+\im X^2)=W(tW^3-Y^3)$ . The morphism $S_6'\to \p^3_{\C(t)}$ which sends $(W:X:Y:Z)$ onto
\begin{center}$\left\{\begin{array}{ll}
 (W^2:WX:WY:Z+\mathbf{i} X^2)& \mbox{ if } W\not=0\mbox{ or }Z\not=-\mathbf{i}X^2 \\
  (W(Z-\mathbf{i} X^2):X(Z-\mathbf{i} X^2):Y(Z-\mathbf{i} X^2):tW^3-Y^3)& \mbox { if } tW^3\not=Y^3\mbox{ or }Z\not=\mathbf{i}X^2  \end{array}\right.$\end{center}contracts the curve $W=0, Z=\im X^2$ onto $(0:0:0:1)$. The image is the smooth cubic 
  
$$S_6=\{(W:X:Y:Z)\in \p^3_{\C(t)}\ |\ Z(WZ-2\im X^2)=tW^3-Y^3\}$$

and the map $S_6'\to S_6$ is the blow-up of $(0:0:0:1)$. 

\bigskip

In order to decide when $S_6,S_7,S_8$ are rational over $\C(t)$ or a finite extension, we will use the following classical result, which is the culmination of several results of V. Iskovskikh and Yu. Manin:
\begin{prop}\label{Prop:NotRational}
Let $k$ be a perfect field and let $S$ be a del Pezzo surface of degree $d$, defined over $k$. If $S$ is minimal $($over $k)$ and $d\le 4$, then $S$ is not rational $($over $k)$.
\end{prop}
\begin{proof}
 If $S$ was rational, there would be a birational map $S\dasharrow \p^2_k$, which decomposes into $\varphi_m\circ \dots\circ \varphi_1$, where the $\varphi_i$ are elementary links (see \cite[Theorem~2.5, page 602]{bib:IskSarkisov}). The classification of the possible links made in \cite[Theorem~2.6, page 604]{bib:IskSarkisov} implies that any link starting from a surface $X$ with $(K_{X})^2\le 4$ gives a surface $X'$ with $(K_{X'})^2\le 4$. This shows that no sequence of links starting from $S$ can go to $\p^2$, and achieves the proof.
\end{proof}

We now prove that $S_6,S_7,S_8$ are minimal over $\C(t)$, which implies that these surfaces are not rational by Proposition~\ref{Prop:NotRational}. We also study the minimal extension which makes the surface becoming rational.
\begin{prop}
Let $K$ be a finite extension of the field $\C(t)$. The surface $S_6$ is rational over $K$ if and only if $K$ contains an element $\rho$ with $\rho^{12}=t$. 

In particular, $S_6$ is not rational over $\C(t)$ and the minimal extension of $\C(t)$  that makes $S_6$ rational is $\C(\sqrt[12]{t})$, which has degree $12$ over $\C(t)$.
\end{prop}
\begin{proof}
The surface $S_6$ is a smooth cubic surface in $\p^3$; it contains thus  exactly $27$ lines  defined over $\overline{\C(t)}$, which are the $27$ $(-1)$-curves on $S_6$. 
Recall that $S_6$ has equation 
$$Z(WZ-2\im X^2)-tW^3+Y^3=0.$$


One checks that the three lines $L_1,L_2,L_3$ of equation
$$Z=0,\ Y=\alpha W,$$
where $\alpha\in \overline{\C(t)}$ is a third root of $t$, are contained in $S_6$.

The other $24$ lines of $S_6$ are the lines $L_\mu$ given by
$$27\im\mu^6(\sqrt{3}+3)W+18X\mu^3+(-9+5\sqrt{3})Z=0$$
$$9\im\mu^2(\sqrt{3}-1)Y+18X\mu^3+2(3-2\sqrt{3})Z=0$$

where $\mu\in  \overline{\C(t)}$ satisfies  $ \mu^{12}=\frac{1}{27}(-5\pm \frac{26}{9}\sqrt{3})t$.

If $K$ contains an element $\rho$ with $\rho^{12}=t$, all $27$ lines of $S_6$ are defined over $K$, which implies that $S_6$ is rational over $K$ (take $6$ disjoint lines and contract them to obtain $\p^2$).

Assume now that $K$ does not contain any $12$-th root of $t$ and let us prove that $S_6$ is not rational over $K$. Any   $\mu\in  \overline{\C(t)}$ satisfying  $ \mu^{12}=\frac{1}{27}(-5\pm \frac{26}{9}\sqrt{3})t$ is conjugate over $K$ to $\xi \mu$ where where $\xi$ is a primitive second or third  root of unity. The lines $L_{\mu}$ and $L_{\xi\mu}$ are thus conjugate by an element of the Galois group $\mathrm{Gal}(\overline{\C(t)}/K)$. One checks that $L_\mu$ and $L_{\xi\mu}$ intersect, which implies that no one of the $24$ curves $L_\mu$ belongs to a contractible curve (defined over $K$).

If one third root of $t$ is contained in $K$, the three lines $L_1,L_2,L_3$ are defined over $K$ but intersect each other at the point $(0:1:0:0)$. After contracting one of the three curves, the del Pezzo surface of degree $4$ that we obtain does not contain any contractible curve defined over $K$ and thus is minimal over $K$. It is therefore not rational by Proposition~\ref{Prop:NotRational}. 

If no third root of $t$ is contained in $K$, the surface $S_6$ is minimal over $K$ and also not rational by Proposition~\ref{Prop:NotRational}.
\end{proof}

\begin{prop}
Let $K$ be a finite extension of the field $\C(t)$. The surface $S_7$ is rational over $K$ if and only if $K$ contains an element $\rho$ with $\rho^{18}=t$. 

In particular, $S_7$ is not rational over $\C(t)$ and the minimal extension of $\C(t)$  that makes $S_7$ rational is $\C(\sqrt[18]{t})$, and has degree $18$ over $\C(t)$.
\end{prop}
\begin{proof}
Recall that $S_7$ has equation $ tW^4=X^3Y+Y^3W+Z^2$ in $\p(1,1,1,2)$. Since $S_7$ is a del Pezzo surface of degree $2$, it has $56$ $(-1)$-curves defined over $\overline{\C(t)}$, corresponding to the $28$ bitangents of the quartic of $\p^2$ given by $$tW^4-X^3Y-Y^3W=0.$$
Recall that classically a bitangent is a curve tangent at two points, but we will allow the case where these two points coincide, and call this "undulatory tangent" again a bitangent. With this definition, every smooth quartic has exactly $28$ bitangents.

We now compute the $28$ bitangents and the $56$ corresponding $(-1)$-curves.

One checks that all lines of the form $aW+bX=0$ are not bitangent to the quartic, which implies that all bitangents are given by $$Y=aW+bX,$$ where $a,b\in \overline{\C(t)}$ are such that $tW^4-X^3(aW+bX)-(aW+bX)^3W$ is a square in $\overline{\C(t)}[W,X]$, that we write $(cX^2+dXW+eW^2)^2$. Any solution of 
\begin{equation}\label{Eq:S7Moins1}
(cW^2+dWX+eX^2)^2=tW^4-X^3(aW+bX)-(aW+bX)^3W
\end{equation}
gives a $(-1)$-curve of $S_7$, of equation $Y=aW+bX$, $Z=cW^2+dWX+eX^2$, and all $(-1)$-curves are obtained like this. Writing Equation~$(\ref{Eq:S7Moins1})$ as
$$(c^2+a^3-t)W^4+(3a^2b+2cd)W^3X+(3ab^2+2ce+d^2)W^2X^2+(a+b^3+2de)WX^3+(b+e^2)X^4=0$$
 every coefficient has to be zero. The last two coefficients yield $b=-e^2$ and $a=-2ed+e^6$.
If $e=0$, one finds $a=b=d=0$ and $c^2=t$, which yields two $(-1)$-curves $L_1,L_2$ of equation
$$Y=0, Z=\pm\sqrt{t}W^2.$$
 We can now assume that $e\not=0$ to find the remaining solutions. The third coefficient yields $c=-\frac{3ab^2+d^2}{2e}=-\frac{d^2-6de^5+3e^{10}}{2e}$. Replacing these in the  first two coefficients, we get respectively 

$$\frac{1}{4e^2}(d^4-44d^3e^5+90d^2e^{10}-60de^{15}+13e^{20}-4e^2t)\mbox{ and }$$
$$\frac{1}{e}(-d^3-6d^2e^5+9de^{10}-3e^{15}).$$

Multiplying the first coefficient by $e(28d+204e^5)$ and adding the second coefficient multiplied by 
$7d^2-299de^5+243e^{10}$ we get $$e(115e^{18}-28t)d-6e^6(11e^{18}+34t)=0.$$
Observe that $115e^8-28t$ is not zero, otherwise we would have $11e^{18}+34t=0$, which is incompatible. We have thus \begin{equation}\label{EqS7d}
d=\frac{6e^5(11e^{18}+34t)}{115e^{18}-28t}.\end{equation}
Replacing this in the second coefficient we get

$$-111e^{14}\frac{e^{54}-29496e^{36}t+401808e^{18}t^2-64t^3}{(115e^{18}-28t)^3},$$
 and find thus $54$ different solutions, which give with the two above the $56$ $(-1)$-curves. In fact, $\frac{e^{18}}{t}$ is a root of the polynomial
 $$Q(X)=X^3-29496X^2+401808X-64,$$
 which has three distinct roots in $\C$, all being real.

If $K$ contains an element $\rho$ with $\rho^{18}=t$, all $56$ $(-1)$-curves of $S_7$ are defined over $K$, which implies that $S_7$ is rational over $K$ (take $7$ disjoint lines and contract them to obtain $\p^2$).

Assume now that $K$ does not contain any $18$-th root of $t$ and let us prove that $S_7$ is not rational over $K$. Any   $\mu\in  \overline{\C(t)}$ satisfying that  $\frac{\mu^{18}}{t}$ is a root of 
$Q(X)$ is conjugate over $K$ to $\xi \mu$ where $\xi$ is a primitive second or third  root of unity. The $(-1)$-curves $L_{\mu}$ and $L_{\xi\mu}$ associated (obtained by setting $e=\mu$ and $e=\xi\mu$)  are thus conjugate by an element of the Galois group $\mathrm{Gal}(\overline{\C(t)}/K)$. As in the case of $S_6$, let us see that the $(-1)$-curves  $L_\mu$ and $L_{\xi\mu}$ intersect. For $\xi=-1$, this is because $L_{-\mu}$ is obtained from $L_\mu$ by replacing $c,d,e$ with $-c,-d,-e$ and letting $a,b$ the same. Any point of intersection of the quartic $tW^4-X^3Y-Y^3W=0$ with the bitangent of equation $Y=aW+bX$ gives a point that belongs to  both $L_\mu$ and $L_{-\mu}$. If $\xi$ is a third root of unity, $L_{\xi\mu}$ is obtained from $\mu$ by replacing $a,b,c,d,e$ with $a,\frac{b}{\xi},c,\frac{d}{\xi},\xi e$ and corresponds to replacing $X$ with $\frac{X}{\xi}$. The intersection of $L_\mu$ with $X=0$ gives one point that belongs to $L_{\xi\mu}$. We have shown that none of the $54$ $(-1)$-curves where $e\not=0$ belongs to a contractible curve (defined over $K$).

The two remaining $(-1)$-curves $L_1,L_2$ of equation $Y=0, Z=\pm\sqrt{t}W^2$ intersect each other in $(0:1:0:0)$. If $\sqrt{t}$ does not belong to $K$, $S_7$ is minimal, and otherwise we can contract one of the two $(-1)$-curves to obtain a minimal del Pezzo surface of degree $3$. In each case $S_7$ is not rational over $K$.
\end{proof}

\begin{prop}
Let $K$ be a finite extension of the field $\C(t)$. The surface $S_8$ is rational over $K$ if and only if $K$ contains an element $\rho$ with $\rho^{30}=t$. 

In particular, $S_8$ is not rational over $\C(t)$ and the minimal extension of $\C(t)$  that makes $S_8$ rational is $\C(\sqrt[30]{t})$, and has degree $30$ over $\C(t)$.
\end{prop}
\begin{proof}
Recall that $S_8$ has equation $ tW^6=X^5W+Y^3+Z^2$ in $\p(1,1,2,3)$. Since $S_8$ is a del Pezzo surface of degree $2$, it has $240$ $(-1)$-curves defined over $\overline{\C(t)}$. Similarly as in the case of $S_7$, these correspond to the hyperplane sections of $\p(1,1,2)$ that cut into two pieces. More precisely, these are the curves of equation
$$Y=aW^2+bWX+cX^2, Z=d W^3+eW^2X+fWX^2+gX^3,$$
where $tW^6-X^5W-(aW^2+bWX-cX^2)^3=(d W^3+eW^2X+fWX^2+gX^3)^2$.
The coefficient of $X^6$ yields $c^3=g^2$, so we can add a new variable $\mu\in \overline{\C(t)}$ with $c=\mu^2$, $g=\mu^3$. The coefficient of $X^5W$ yields $1+3bc^2=2gf$, so $1+3\mu^4b=2\mu^3 f$, which implies that $\mu\not=0$ and $f=\frac{1+3\mu^4b}{2\mu^3}$. We replace these in the coefficient of $X^4W^2$ and find $e=\frac{12\mu^{10}a+3\mu^8b^2-6\mu^4b-1}{8\mu^9}$, then find $d=\frac{12\mu^{14}ab-\mu^{12}b^3-12\mu^{10}a+15\mu^8b^2+9\mu^4b+1}{16\mu^{15}}$ with the coefficient of $X^3W^3.$ 

The coefficients of $X^2W^4$ and $XW^5$ become two polynomials of degree $2$ in $a$ (and higher degree in $b,\mu$), so one can make a linear combination to cancel the terms of degree $2$ and find a linear equation in $a$, which yields
$$a=\frac{10b^4\mu^{16}+85b^3\mu^{12}+90b^2\mu^8+25\mu^4 b+2}{30\mu^{10}(b^2\mu^8+4\mu^4b+1)}$$ (one checks that $b^2\mu^8+4\mu^4b+1=0$ is incompatible with the two equations we obtained). 

Replacing the value of $a$ in the coefficients of $X^2W^4$ and $XW^5$ one sees that these are zero if and only if one of the following two polynomials vanishes:

$$P_1=5b^4\mu^{16}-690\mu^{12}b^3-260\mu^8 b^2-30\mu^4b-1$$
$$P_2=5b^4\mu^{16}+10\mu^{12}b^3-20\mu^8b^2-10\mu^4b-1$$

For each of the polynomials, one performs successive  polynomial divisions between $P_i$ and the coefficient of $W^6$ (which has degree $9$ in $b$) to obtain a polynomial of degree $1$ in $b$ and obtain then $b$ in terms of $\mu$. One gets respectively 

$b=\frac{-5(16307084980800\mu^{90}t^3-60864048645838405658640\mu^{60}t^2+1761869851700383404t\mu^{30}-2251428325403)}
{2\mu^4(198455329800000\mu^{90}t^3-740708401360188117142800\mu^{60}t^2+20921826963788922780t\mu^{30}-40377544164371)}$

or

$b=\frac{-10(224784123775200\mu^{90}t^3+2858233826211840\mu^{60}t^2+7607560177676934\mu^{30}t-9692094622039483)}{\mu^4(15103883194560000\mu^{90}t^3+194657569344061200\mu^{60}t^2+526365630369285480\mu^{30}t-667567630291039199)}$
and the remaining equation for $\mu$ that we obtain is  respectively
$$108000\mu^{30}t(5400\mu^{90}t^3-20154789349200\mu^{60}t^2+522900235t\mu^{30}+1254)+1=0,$$
$$108000\mu^{30}t(5400\mu^{90}t^3-10810800\mu^{60}t^2-44551045t\mu^{30}-611864)+1=0.$$
So $\mu^{30}t$ is root of one of the polynomials
$$Q_1(X)=108000X(5400X^3-20154789349200X^2+522900235X+1254)+1,$$
$$Q_2(X)=108000X(5400X^3-10810800X^2-44551045X-611864)+1,$$
which have both $4$ distinct roots in $\C$, all being real. This yields the $240$ $(-1)$-curves of the surface $S_8$.

If $K$ contains an element $\rho$ with $\rho^{30}=t$, all $240$ $(-1)$-curves of $S_8$ are defined over $K$, which implies that $S_8$ is rational over $K$ (take $8$ disjoint lines and contract them to obtain $\p^2$).

Assume now that $K$ does not contain any $30$-th root of $t$ and let us prove that $S_8$ is not rational over $K$. Any   $\mu\in  \overline{\C(t)}$ satisfying that  $\mu^{30}{t}$ is a root of 
one of the two polynomials $Q_1,Q_2$ is conjugate over $K$ to $\xi \mu$ where $\xi$ is a primitive second, third or fifth  root of unity. The $(-1)$-curves $L_{\mu}$ and $L_{\xi\mu}$ associated   are thus conjugate by an element of the Galois group $\mathrm{Gal}(\overline{\C(t)}/K)$. As in the cases of $S_6$, $S_7$, let us see that the $(-1)$-curves  $L_\mu$ and $L_{\xi\mu}$ intersect. 

Note that $L_{\xi\mu}$ is obtained from $L_\mu$ by replacing \begin{center}$a,b,c,d,e,f,g,\mu$ with $\frac{a}{\xi^{10}},\frac{b}{\xi^{4}},c\xi^2,\frac{d}{\xi^{15}},\frac{e}{\xi^9},\frac{f}{\xi^3},g\xi^3$.\end{center}
Remembering that $L_\mu$ is given by
$$Y=aW^2+bWX+cX^2, Z=d W^3+eW^2X+fWX^2+gX^3,$$
 $L_{\xi\mu}$ is obtained by replacing $W$ and $X$ with $\frac{W}{\xi^5}$ and $\xi X$, which corresponds to the automorphism
 $$\theta_\xi\colon (W:X:Y:Z)\mapsto (\frac{W}{\xi^5}:\xi X:Y:Z)=(\frac{W}{\xi^6}:X:\frac{Y}{\xi^2}:\frac{Z}{\xi^3}).$$
 Any point of $L_\mu$ that is fixed by $\theta_\xi$ is also contained in $L_{\mu\xi}$. It suffices to show that this point exists to see that $L_\mu$ and $L_{\mu\xi}$ intersect. If $\xi^5=1$, we cut $L_\mu$ with $X=0$ and get one point. If $\xi^3=1$, $\theta_\xi((W:X:Y:Z))=(\xi W:\xi X:Y:Z)=(W:X:Y\xi:Z)$, so we cut $L_\mu$ with $Y=0$ and get two points (or one with multiplicity $2$). If $\xi=-1$, $\theta_\xi((W:X:Y:Z))=(-W:-X:Y:Z)=(W:X:Y:-Z)$, so we cut $L_\mu$ with $Z=0$ and get three points (or less, with multiplicity).
 
 In each case, the fact that $L_\mu$ and $L_{\mu\xi}$ intersect implies that no one of the $240$ curves $L_\mu$ belongs to a contractible curve (defined over $K$). The surface $S_8$ is therefore minimal over $K$ and not rational  by Proposition~\ref{Prop:NotRational}.\end{proof}

\subsection{Conic bundle case}\label{Sec:CB}
Note that the $x$-projection $\A^3\to \A^1$ restricted to the cases $A_n$ and $D_n$ gives  a fibration where a general fibre is an affine conic. We will extend this and obtain a natural compactification as a conic bundle, so that any fibre is a projective conic, which can be smooth or the union of two transerval lines (in this case the fibre is a singular fibre, and only finitely many of them occur). We can then embedd this conic bundle over the affine line $\A^1_{\C(t)}$ into a projective surface being a conic bundle over the projective line $\p^1_{\C(t)}$. 

We can therefore suppose that the conic bundle $\pi\colon S\to \p^1_{\C(t)}$ is \emph{minimal}. Similary as for surface, a conic bundle $(S,\pi)$ is minimal if any birational morphism $\varphi\colon S\to S'$, where $(S',\pi')$ is another conic bundle satisfying $\pi=\pi'\circ \varphi$, is an isomorphism.

It follows from this definition that a conic bundle $\pi\colon S\to \p^1_{\C(t)}$ is minimal if and only if there is no contractible curve on $S$ which is contained in a finite number of fibres. Indeed, any birational morphism between two conic bundles which is not a isomorphism contracts a finite number of contractible curves, all contained in a finite number of fibres.

Every smooth fibre is not contractible, but any component $f_1$ of a singular fibre $f=f_1\cup f_2$ is contractible over $\overline{\C(t)}$; its orbit by the Galois group $\mathrm{Gal}(\overline{\C(t)}/\C(t))$ is then contractible if and only it does not contain the component $f_2$.

In order to decide when the surface obtained is rational over $\C(t)$ or a finite extension, we will use the following classical result, which is the analogue of Proposition~\ref{Prop:NotRational} for conic bundles.

\begin{prop}\label{Prop:ConicBundlesNotRational}
Let $k$ be a perfect field and let $S$ be a smooth projective surface that admits a conic bundle structure $\pi\colon S\to \p^1_k$. Suppose that $(S,\pi)$ is minimal $($over $k)$ and that the number of singular fibres of $\pi$ is $d\ge 0$.

\begin{enumerate}
\item
If $d\le 1$ and if there exists a point $p\in S$ defined over $k$ then $S$ is rational $($over $k)$.
\item
If $d\ge 4$, then $S$ is not rational $($over $k)$.
\end{enumerate}
\end{prop}
\begin{proof}
First assume that $d\le 1$ and that there exists a point $p\in S$ defined over $k$. In fact, $d=0$ is the only possibility because of the minimality of $(S,\pi)$. The surface is a line bundle defined, so is equivalent over $\overline{k}$ to an Hirzebruch surface $\mathbb{F}_n$. The point $p$ belongs to a fibre defined over $k$, which contains points that do not lie on the exceptional section. Blowing-up one of this point and contracting the transform of the fibre gives rise to a Hirzebruch surface $\mathbb{F}_{n-1}$ if $n\not=0$ and $\mathbb{F}_{n+1}$ otherwise. By induction we can go to $\mathbb{F}_1$; the unique exceptional section is then defined over $k$ and we can contract it to go to $\mathbb{P}^2$, so $S$ is rational.

Suppose now that $d\ge 4$, which implies that $(K_S)^2=8-d\le 4$, and let us show that $S$ is not rational. The proof is the same as in Proposition~\ref{Prop:NotRational}.  If $S$ was rational, there would be a birational map $S\dasharrow \p^2_k$, which decomposes into $\varphi_m\circ \dots\circ \varphi_1$, where the $\varphi_i$ are elementary links (see \cite[Theorem~2.5, page 602]{bib:IskSarkisov}). The classification of the possible links made in \cite[Theorem~2.6, page 604]{bib:IskSarkisov} implies that any link starting from a surface $X$ with $(K_{X})^2\le 4$ gives a surface $X'$ with $(K_{X'})^2\le 4$. This shows that no sequence of links starting from $S$ can go to $\p^2$, and achieves the proof.
\end{proof}

{\bf Case $A_n$ -- } 
The generic fibre of $F$ is the affine surface $$\{(x,y,z)\in \A^{3}_{\C(t)}\ |\  x^n-yz=t\},$$ defined over $\C(t)$, which naturally embeds into the quasi-projective surface
$$U=\{(x,(w:y:z))\in \A^1_{\C(t)}\times \p^2_{\C(t)} \ |\ x^nw^2-yz=tw^2\}$$
via the map $(x,y,z)\mapsto (x,(1:y:z))$. The projection $U\to \A^1_{\C(t)}$ yields a conic bundle. Any general fibre is a smooth conic, and there are exactly $n$ singular fibres over $\overline{\C(t)}$, consisting of two intersecting lines, above the points $x^n=t$. We can compactify $U$ to a projective surface $X$ defined over $\C(t)$ which has a conic bundle structure $\pi\colon X\to \p^1_{\C(t)}$ extending the conic bundle structure on $U$; the singular fibres are then two $(-1)$-curves intersecting into one point. 

The curve $y=0$, $x^n=t$ is defined over $\C(t)$ and is thus contractible (it is the union of $n$ $(-1)$-curves which belong to the same orbit of the Galois group of $\overline{C(t)}/\C(t)$). This leads to a projective conic bundle with at most one singular fibre, the fibre at infinity that we do not see on $U$. This one is  rational by Proposition~\ref{Prop:ConicBundlesNotRational}.

\bigskip

{\bf Case $D_n$ -- } 
The generic fibre of $F$ is the affine surface $$\{(x,y,z)\in \A^{3}_{\C(t)}\ |\  x^{n-1}+xy^2+z^2=t\},$$ defined over $\C(t)$, which naturally embeds into the quasi-projective surface
$$U_n=\{(x,(w:y:z))\in \A^1_{\C(t)}\times \p^2_{\C(t)} \ |\ x^{n-1}w^2+xy^2+z^2=tw^2\}$$
via the map $(x,y,z)\mapsto (x,(1:y:z))$. As before, the projection $U\to \A^1_{\C(t)}$ yields a conic bundle. Any general fibre is a smooth conic, and there are exactly $n$ singular fibres over $\overline{\C(t)}$, consisting of two intersecting lines, above the points $x=0$ and $x^{n-1}=t$.

The difference with the case $A_n$ is that now that the singular fibres do not necessarily contain contractible curves, since there are component that are conjugate by the Galois group to curves that they touch. More precisely we have the following:

\begin{prop}
Let $K$ be a finite extension of the field $\C(t)$, and let $n\ge 4$. We write $2n-2=ab$ where $a,b$ are integers, $a$ is a power of two and $b$ is odd. The surface $U_n$ is rational over $K$ if and only if $K$ contains an element $\rho$ with $\rho^{a}=t$. 

In particular, $U_n$ is not rational over $\C(t)$ and the minimal extension of $\C(t)$  that makes $U_n$ rational is $\C(\sqrt[a]{t})$, and has degree $a$ over $\C(t)$.
\end{prop}
\begin{proof}
 We can again compactify $U_n\to \mathbb{A}^1_{\C(t)}$ to a conic bundle $X_n\to \p^1_{\C(t)}$, adding only one curve which can be smooth or the union of two intersecting lines. We can moreover assume that this curve is irreducible over $\C(t)$, otherwise we contract one of the two components.

Let us study the different singular fibres (over $\overline{\C(t)}$). 
The fibre over $x=0$ is the conic $z^2=t w^2$, which decomposes into the two lines $z=\pm\sqrt{t}w$. The fibres over $x^{n-1}=t$ are $n-1$ singular fibres given by $x^{n-1}=t$, $y^2x+z^2=0$. These decompose into $2n-2$ $(-1)$-curves corresponding to $x=\mu^2$, $z=\im y \mu$, where $\mu\in \overline{\C(t)}$ is a root of the polynomial $Q(X)=X^{2(n-1)}-t$. The last possible singular fibre is the fibre at infinity, that we do not see on $U_n$. We will not  need to study it.

Suppose first that there exists $\rho\in K$ with $\rho^{a}=t$. Since $a\ge 2$ is even, each of the two components of the fibre $x=0$ is defined over $K$. Denoting by $\xi\in \C$ a primitive $a$-th root of unity, the polynomial $Q(X)$ decomposes in $K[X]$ into $Q(X)=\prod_{i=1}^a (X^b-\rho \xi^i)$. The $2(n-1)$ components of the singular fibres over non-zero fibres decompose then into an union of $a$ curves $C_1,\dots,C_a$ of $b$ components, where $C_i$ is a curve define over $K$ given by 
$x^b=(\rho\xi^i)^2$, $z^b=\im ^b y^b(\rho\xi^i)$ for each $i=1,\dots,a$. Note that $C_i$ is the union of $b$ disjoint $(-1)$-curves, and that $C_i$ intersect $C_j$ if and only if $i-j=\pm \frac{a}{2}$. We can then contract the curves $C_1,C_2,\dots,C_{\frac{a}{2}}$ and the curve $x=0$, $w=\rho^{\frac{a}{2}}w$, to obtain a minimal conic bundle with zero or one singular fibres, which is rational by Proposition~\ref{Prop:ConicBundlesNotRational}.

Suppose now that $K$ does not contain any $a$-th root of $t$, and let us prove that $U_n$ is not rational over $K$. Since $a$ is a power of $2$, $b$ is odd and $Q(x)=x^{ab}-t$, any $\mu\in \overline{\C(t)}$ satisfying $Q(\mu)=0$ is conjugate over $K$ to $- \xi$. This means that the $(-1)$-curve $x=\mu^2$, $z=\im y \mu$ touches its conjugate $x=\mu^2$, $z=-\im y \mu$. Contracting the curves over $x=0$ or at infinity if needed, we obtain a minimal conic bundle over $K$ with at least $n-1\ge 3$ singular fibres, and which is therefore not rational if $n\ge 5$ by Proposition~\ref{Prop:ConicBundlesNotRational}. It remains to study the case $n=4$, which yields $2n-2=6=ab$, so $a=2$ and $b=3$. The fact that $K$ does not contain any $a$-root of $t$ implies that the fibre over $x=0$ does not contain any contractible curve, so we get a minimal conic bundle over $K$ with at least $4$ singular fibres, which is thus  not rational by Proposition~\ref{Prop:ConicBundlesNotRational}.
\end{proof} 
\section{Automorphisms of Klein surfaces}
In this section, we determine the automorphisms of the Klein surfaces given by the equation $f=d_n,e_6,e_7$ or $e_8$ in $\C^3$. Algebraically, it corresponds to the group of $\C$-automorphisms of the ring $\C[x,y,z]/(f)$. The symmetries of the equations yield obvious automorphisms, but there is a priori no reason that some complicated non-linear automorphisms exist. We will show that this holds only in the case of $a_n$, where the group of automorphisms is an amalgamated product, a result already known.

We will embedd the affine surface $$S_f=\{(x,y,z)\in \C^3\ |\ f(x,y,z)=0\}$$ into a normal projective surface $X$, such that any point of $X\backslash S_f$ is a smooth point of $X$. Any automorphism of $S$ corresponds to a birational map $\varphi\colon X\dasharrow X$ which is biregular on the points of $S_f$. If $\varphi$ is not an automorphism of $X$, we denote by $\eta\colon Z\to X$ the blow-up of all base-points of $\varphi$ (that may lie on $X$ or being infinitely near). The map $\pi\colon Z\to X$ given by $\pi=\varphi\eta$ is a birational morphism, so we obtain the following commutative diagram

\begin{equation}\label{Eq:DiagrammFond}
\xymatrix{
& Z\ar[ld]_{\eta}\ar[rd]^{\pi}\\
X\ar@{-->}[rr]_{\varphi}&&X,
}\end{equation}
where $\eta,\pi$ are blow-ups of points of $X$ (or infinitely near) that do not belong to $S_f$.

 The nature of the boundary $B=X\backslash S_f$ will impose conditions on the diagram above. We first deal with the case $e_6,e_7,e_8$, and then study $d_n$.

\subsection{Case $e_6$, $e_7$, $e_8$}

As in $\S\ref{Sec:dPdf}$, if $f=e_6,e_7,e_8$, the affine complex surface $S_f\subset \C^3$ embedds into a projective surface $X$ contained in a weighted projective space. The surface $X$ is a singular del Pezzo surface, smooth outside $S_f$. We obtain the surface by replacing $t$ with $0$ in the surfaces $S_6,S_7,S_8$ of $\S\ref{Sec:dPdf}$, and get the following complex surfaces (for $e_6$, we  take $S_6'$ instead of $S_6$ to keep more symmetry) :

$$\mathcal{S}_6=\{(W:X:Y:Z)\in \p(1,1,1,2)_{\C}\ |\ X^4+Y^3W+Z^2=0\},$$
$$\mathcal{S}_7=\{(W:X:Y:Z)\in \p(1,1,1,2)_{\C}\ |\ X^3Y+Y^3W+Z^2=0\},$$
$$\mathcal{S}_8=\{(W:X:Y:Z)\in \p(1,1,2,3)_{\C}\ |\ X^5W+Y^3+Z^2=0\}.$$

\begin{prop}\label{Prop:ExtensionAuto}
If $f=e_i$ for $i\in\{6,7,8\}$, any automorphism of $S_f$ extends to an automorphism of $\mathcal{S}_i$.
\end{prop}
\begin{proof}
The projective surface $X=\mathcal{S}_i$ is a singular del Pezzo surface of degree respectively $2,2,1$. It can thus be obtained by blowing-up $7,7,8$ points respectively in $\p^2$, obtaining a weak del Pezzo and by taking the plurianti-canonical morphism, which contracts the $(-2)$-curves onto the singular point $(1:0:0:0)$. The Klein surface is the complement in $\mathcal{S}_i$ of the curve $B$ of equation $W=0$, this latter curve being equivalent to the anti-canonical divisor; in particular $B^2$ is respectively $2,2,1$. 

On $\mathcal{S}_6$, the curve $B$ is the pull-back by the double covering $\mathcal{S}_6\to \p^2$ of a quadritangent to the quartic of ramification, so $B$ is the union of two $(-1)$-curves intersecting into one point, but with multiplcity two. We can check that  $B^2=(E_1+E_2)^2=(E_1)^2+(E_2)^2+2E_1E_2=-1-1+4=2$.

On $\mathcal{S}_7$, the curve $B$ is irreducible; it is a rational curve with one cuspidal point which is $(0:0:1:0)$. Since the curve is anti-canonical, it has arithmetic genus $1$ (comes from a cubic of $\p^2$),  so the point is a simple cusp of multiplicity $2$.

On $\mathcal{S}_8$, the curve $B$ is once again a rational curve, with one cuspidal point of multiplicity $2$.

Let us suppose the existence of an automorphism of $S_f$ that extends to a  birational map $\varphi\colon X\dasharrow X$ which is not an isomorphism. We obtain the diagram (\ref{Eq:DiagrammFond}) as above. The first curve contracted by $\pi$ is the strict transform of a component of $B$ by~$\eta^{-1}$, and is a $(-1)$-curve on $Z$.

 In the case where $i=7,8$, the curve $B$ has one component, which is singular and of self-intersection $2$ or $1$. Its strict transform on $Z$ being smooth, the morphism $\eta$ blows-up the singular point of $B\subset \mathcal{S}_i$. But the self-intersection of the strict transform is then $\le -2$, which yields a contradiction.
 
 If $i=6$, the curve $B$ has two components that we write as before $E_1$ and $E_2$. We can assume that the strict transform $\widetilde{E_1}\subset Z$ of $E_1\subset X$ is the first curve contracted by $\pi$. Since $E_1$ is a $(-1)$-curve on $\mathcal{S}_6$, the morphism $\eta$ does not blow-up any point of $E_1$. In particular, the strict transform $\widetilde{E_2}\subset Z$ of $E_2\subset X$  intersects again $\widetilde{E_1}$ into a point with multiplicity $2$. After contracting $\widetilde{E_1}$, the curve $\widetilde{E_2}$ becomes singular, so we get a singular curve in the boundary of $X=\mathcal{S}_6$, which is impossible since $B$ is the union of $E_1,E_2$, both smooth.
\end{proof}
\begin{coro}\label{Coro:KleinAuto}
Any automorphism of the Klein surfaces 
$$S_{e_6}=\{(x,y,z)\in \C^3\ |\ x^4+y^3+z^2=0 \}$$
$$S_{e_7}=\{(x,y,z)\in \C^3\ |\ x^3y+y^3+z^2=0 \}$$
$$S_{e_8}=\{(x,y,z)\in \C^3\ |\ x^5+y^3+z^2=0 \}$$
extends to a diagonal automorphism of $\C^3$ of the the form $(x,y,z)\mapsto (\alpha x,\beta y,\gamma z)$.
In particular, $\Aut(S_{e_6})\cong \C^{*}\times \{\pm 1\}$, $\Aut(S_{e_7})\cong \Aut(S_{e_8})\cong \C^{*}$.
\end{coro}
\begin{proof}
For $i=6,7,8$, any automorphism of the affine surface 
$$S_{e_i}=\{(x,y,z)\in \C^3\ | e_i(x,y,z)=0\}$$ extends to an automorphism of the projective surface $\mathcal{S}_i$ (Proposition~\ref{Prop:ExtensionAuto}). The group of automorphisms of the affine surface is thus equal to the group of automorphisms of $\mathcal{S}_i$ that preserve the curve of equation $W=0$. The embedding in the weighted projective space being (anti)-canonical, the automorphisms of $\mathcal{S}_i$ come from automorphisms of the weighted projective space.

In the case of $i=6,7$, the surface $\mathcal{S}_i$ is in $\p(1,1,1,2)$, and any automorphism of $\p(1,1,1,2)$ that preserves $W=0$ is of the form
$$(W:X:Y:Z)\mapsto (W:aW+bX+cY:dW+eX+fY:g Z+P(W,X,Y))$$
where $a,b,c,d,e,f,g \in \C$ and $P$ is a polynomial of degree $2$ in $W,X,Y$. The form of the equation being $Z^2+Q(W,X,Y)$, we directly see that $P=0$. The singular point $(1:0:0:0)$ being fixed, we find $c=f=0$. We can then either check directly, or use the singularity to remove coefficients, that the automorphism is diagonal, i.e.\ of the form 
$$(W:X:Y:Z)\mapsto (W:bX:fY:g Z).$$
The same argument apply to case $i=8$. 

The group $\Aut(S_{e_6})$, $\Aut(S_{e_7})$, $\Aut(S_{e_8})$ corresponds thus to automorphisms of the form $(x,y,z)\mapsto (\alpha x,\beta y,\gamma z)$ which preserve the equation. These are subgroups of $(\C^{*})^3$ given respectively by 
$$\Aut(S_{e_6})\cong \{(\alpha,\beta,\gamma)\in (\C^{*})^{3}\ |\ \alpha^4=\beta^3=\gamma^2\},$$
$$\Aut(S_{e_7})\cong \{(\alpha,\beta,\gamma)\in (\C^{*})^{3}\ |\ \alpha^3\beta=\beta^3=\gamma^2\},$$
$$\Aut(S_{e_8})\cong \{(\alpha,\beta,\gamma)\in (\C^{*})^{3}\ |\ \alpha^5=\beta^3=\gamma^2\}.$$
These are thus isomorphic respectively to $\C^{*}\times \{\pm 1\}$, $\C^{*}$, $\C^{*}$ via the maps
$$(t,\epsilon)\mapsto (t^3,t^4,\epsilon t^6), t\mapsto (t^4,t^6,t^9), t\mapsto (t^6,t^{10},t^{15}).$$
\end{proof}

\subsection{Case $d_n$}
In the case of $d_n$, the projective normal surface $X$ will be an hypersurface of a $\p^2$-bundle over $\p^1$. 

Using the notation of \cite{bib:Freu}, the $\p^2$-bundle $F_{a,b}$ is $\mathbb{P}(\mathcal{O}_{\p^1}\oplus\mathcal{O}_{\p^1}(a)\oplus\mathcal{O}_{\p^1}(b))$, and can be viewed as the glueing of $U_{a,b,0}=\p^2\times \C$ and $U_{a,b,\infty}=\p^2\times \C$ along $\p^2\times \C^{*}$, where the identification map is given by the involution $$((w:y:z),x) \dasharrow ((w:x^{-a}y:x^{-b}z),\frac{1}{x}).$$ It is a generalisation of the construction of Hirzebruch surfaces. The $\p^2$-bundle is given by the map $ F_{a,b}\to \p^1$ corresponding to $((w:y:z),x)\to (x:1)$ in the first map and $((w:y:z),x)\to (1:x)$ in the second one.

If $n=2k$, we take the $\p^2$-bundle $F_{k,k}=\mathbb{P}(\mathcal{O}_{\p^1}\oplus\mathcal{O}_{\p^1}(k-1)\oplus\mathcal{O}_{\p^1}(k-1))$, and denote by $\mathcal{D}_n$ the projective surface that restrict to the following surfaces on each chart:

$$\begin{array}{llcll}
\{&((w:y:z),x)\in U_{k-1,k-1,0} &|& x^{n-1}w^2+xy^2+z^2&\}\\
\{&((w:y:z),x)\in U_{k-1,k-1,\infty }& |& w^2+y^2+xz^2&\}\end{array}$$

If $n=2k+1$, we take the $\p^2$-bundle $F_{k-1,k}=\mathbb{P}(\mathcal{O}_{\p^1}\oplus\mathcal{O}_{\p^1}(k-1)\oplus\mathcal{O}_{\p^1}(k))$, and denote by $\mathcal{D}_n$  the projective surface that restrict to the following surfaces on each chart:

$$\begin{array}{llcll}
\{&((w:y:z),x)\in U_{k-1,k,0} &|& x^{n-1}w^2+xy^2+z^2&\}\\
\{&((w:y:z),x)\in U_{k-1,k,\infty }& |& w^2+xy^2+z^2&\}\end{array}$$

In each case, we embedd the surface $S_{d_n}$ in the first affine chart of $\mathcal{D}_n$, via the embedding $$(x,y,z)\mapsto ((1:y:z),x)$$ of $\C^3$ into $\p^2\times \C$, and see that $(0,0,0)$ is sent onto the unique singular point of $\mathcal{D}_n$. 
The $\p^2$-bundle $F_{k,k}\to \p^1$ or $F_{k-1,k}\to \p^1$ restricts to a morphism $\rho\colon \mathcal{D}_n\to \p^1$, which has general fibres isomorphic to a smooth conic (or to $\p^1$), and which has two special fibres, namely the fibre of $(1:0)$ which is the union of two transversal lines, and the fibre of $(0:1)$, which is a double line passing through the  singular point of $\mathcal{D}_n$.
The complement of $\mathcal{D}_n\backslash S_{d_n}$ consists of the curve $C_n$ of equation $w=0$ in each chart, and of the curve $F$ of equation $x=0$ in the second chart, corresponding to the fibre of $(1:0)$. The curve $F$ decomposes into two curves $F_+,F_{-}$ of equation $x=0,w=\pm\im z$ if $n$ is even and $x=0,w=\pm \im y$ if $n$ is odd.

The precise description of the boundary $\mathcal{D}_n\backslash S_{d_n}=C_n\cup F_{+}\cup F_{-}$, given in Lemma~\ref{Lemm:DnBoundary} below, will yield the structure of the automorphisms of $S_{d_n}$.
\begin{lemm}\label{Lemm:DnBoundary}
The complement of $S_{d_n}$ in $\mathcal{D}_n$ is the union of the three curves $C_n$, $F_{+}$, $F_{-}$, all being isomorphic to $\p^1$. Any two of them intersect transversally, into exactly one point, which is $C_n\cap F_{+}\cap F_{-}$. Moreover, $(C_n)^2=3-n$.

\begin{center}\begin{pspicture}(-1,-1)(7,1)
\psline[linecolor=gray](2.97,-1)(2.97,1)
\psline[linecolor=gray](3.03,-1)(3.03,1)
\psline[linecolor=gray](1.5,-1)(1.5,1)
\psline[linecolor=gray](2,-1)(2,1)
\psline[linecolor=gray](2.5,-1)(2.5,1)
\psline[linecolor=gray](1,-1)(1,1)
  \pscurve[linewidth=0.05](2.7,0)(0,-0.3)(5,0)(6,-0.3)
\psline[linewidth=0.05](4,-1)(6,1)
\psline[linewidth=0.05](4,1)(6,-1)
\rput(3,1){\textcolor{gray}{\textbullet}}
\rput(-0.3,-0.2){$C_n$}
\rput(-0.4,-0.5){\scriptsize $[3-n]$}
\rput(4,0.75){$F_{+}$}
\rput(4,0.4){\scriptsize $[-1]$}
\rput(6.2,0.75){$F_{-}$}
\rput(6.2,0.4){\scriptsize $[-1]$}
\end{pspicture}\\
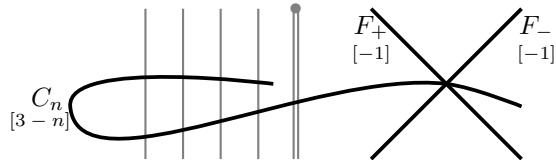
{\it \scriptsize The situation on the surface $\mathcal{D}_n$.\\ Curves in bold are the boundary $\mathcal{D}_n\backslash S_{d_n}$ and curves in grey are fibres.}\end{center}\end{lemm}
\begin{proof}
The only assertion which does not directly follow from the description above is  the self-intersection of the curve $C_n$.

As before, we choose $k$ so that $n\in \{2k,2k+1\}$. Let $D_n\subset \mathcal{D}_n$ be the curve given by $y=0$ on each chart. If $n=2k$, the two curves $D_n$ and $C_n$ intersect at one point, being $((0:0:1),0)\in U_{k-1,k_1,\infty}$, with distinct tangent directions, so $C_n\cdot D_n=1$; if $n=2k+1$ the two curves are disjoint so $D_n\cdot C_n=0$. In both cases, we can say that $C_n\cdot D_n=2k+1-n$. We use the rational map $g\in \C(\mathcal{D}_n)^{*}$ given by $w/y$ on the second chart and thus by  $w/(yx^{1-k})=wx^{k-1}/y$ on the first chart. The principal divisor associated is $C_n-D_n+(k-1)F_0$, where $F_0$ is the fibre of $(0:1)$, linearly equivalent to $F$, the fibre of $(1:0)$. Since $C_n\cdot F=2$, we can compute $$(C_n)^2=C_n\cdot (D_n-(k-1)F)=(2k+1-n)-(k-1) (D_n \cdot F)=3-n.$$
\end{proof}

\begin{prop}\label{Prop:ExtensionAutoDn}
For any $n\ge 4$, any automorphism of $S_{d_n}$ extends to an automorphism of $\mathcal{D}_{n}$.
\end{prop}
\begin{proof}
Let us suppose the existence of an automorphism of $S_{d_n}$ that extends to a  birational map $\varphi\colon \mathcal{D}_n\dasharrow \mathcal{D}_n$ which is not an isomorphism. 

We obtain the diagram (\ref{Eq:DiagrammFond}), with some birational morphisms $\pi\colon Z\to X=\mathcal{D}_n$ and $\eta\colon Z\to X=\mathcal{D}_n$ such that $\pi\eta^{-1}=\varphi$.

The first curve contracted by $\pi$ is the strict transform  by~$\eta^{-1}$ of a component of the boundary $B$ and is a $(-1)$-curve on $Z$. The boundary $B$ is equal to $C\cup F_{+},F_{-}$, which have self-intersection $3-n$, $-1$, $-1$ respectively (Lemma~\ref{Lemm:DnBoundary}). In particular, the first curve contracted by $\pi$ is the strict transform of $F_{+}$ or $F_{-}$, say $F_{+}$. Since $F_{+}$ has self-intersection $-1$ on $X$ and its strict transform on $Z$ also, the point of intersection of $F_{+},F_{-1}$ and $C$ is not blown-up by $\eta$. After contracting the strict transform of $F_{+}$,  the strict transforms of $C_n$ and $F_{-}$ become tangent and it is not possible to obtain the boundary $F_{+},F_{-},C_n$ after other contractions.
\end{proof}

\begin{coro}\label{Coro:KleinAutoDn}
For $n\ge 4$, the group 
$$G=\{(x,y,z)\mapsto (\lambda^2 x,\pm \lambda^{n-2}y,\pm \lambda^{n-1}z)\ |\ \lambda\in \C^{*}\}\cong \C^{*}\times \Z{2}$$
acts on the Klein surface
$$S_{d_n}=\{(x,y,z)\in \C^3\ |\ x^{n-1}+xy^2+z^2 \}.$$

If $n\ge 5$, $\Aut(S_{d_n})=G$.

If $n=4$, $\Aut(S_{d_n})=G\rtimes <\tau>$, where $\tau$ is the automorphism of order $3$ given by 

$$\tau\colon (x,y,z)\mapsto \left(-\frac{1}{2}x+\frac{\im}{2}y,\frac{3}{2}\im x-\frac{1}{2} y,z\right).$$
\end{coro}
\begin{proof}

For $n\ge 4$, any automorphism of the affine surface $S_{d_n}$ extends to an automorphism of the projective surface $\mathcal{D}_n$ (Proposition~\ref{Prop:ExtensionAutoDn}), which will therefore preserves the boundary $\mathcal{D}_n\backslash S_{d_n}=F_{+}\cup F_{-}\cup D_n$. 
Denote by $G\subset \Aut(\mathcal{D}_n)$ the normal subgroup of automorphisms that preserves the conic bundle structure $\mathcal{D}_n\to \p^1$, i.e.\ that sends any fibre on another fibre, and by $G_0\subset G$ the subgroup of automorphisms which acts trivially on the basis, i.e.\ which leaves any fibre invariant.

We first prove that $G_0\cong (\Z{2})^2$, and is generated by $\sigma_y\colon (x,y,z)\mapsto (x,-y,z)$ and $\sigma_z\colon (x,y,-z)$. Let $g$ be an element of $G_0$. For any $x_0\in \C$,  $g\in G_0$ restricts to an automorphism of the fibre of $(x_0:1)$ (points where $x=x_0$). This curve being a smooth conic for $x_0\not=0$, the automorphism extends to a linear automorphism of the plane that contains it, which preserves the two points of $C_n$ and thus the line of equation $w=0$ which contains these. 
In particular we can write $g$ as
$$(x,y,z)\mapsto (x,ay+bz+e,cy+dz+f),$$
where $a,b,c,d,e,f\in \C[x]$ and $ad-bc\not=0$, and for any $x\not=0$, the automorphism is affine and preserves the equation of the conic, so

\begin{equation}x^{n-1}+x(ay+bz+e)^2+(cy+dz+f)^2=\lambda(x^{n-1}+xy^2+z^2)\label{Eqxy}\end{equation}
 for some $\lambda\in \C[x]^{*}$. Computing the coefficients of $wy$ and $wz$ we find $2xae+2cf=0$ and $2xbe+2df=0$. Subtracting $b$ times the first one from $a$ times the second one, we find $2(bc-ad)f=0$, which implies that $f=0$. It implies that $2xae=0$ and $2xbe=0$, which yields $e=0$. 
 Comparing the constant terms of Equation~(\ref{Eqxy}) we find $x^{n-1}=\lambda x^{n-1}$, hence $\lambda=1$. The coefficients of $y^2$ and $z^2$ yield respectively:
 \begin{eqnarray}
xa^2+c^2-x&=&0\label{eqx2}\\
xb^2+d^2-1&=&0\label{eqx3}\end{eqnarray}
Equation $(\ref{eqx2})$ implies that $a,c$ are constant polynomials (otherwise the highest degree of $xa^2+c^2$ would not vanish). It yields thus $a^2=1$ and $c=0$. Equation $(\ref{eqx3})$ gives similarly $b=0$ and $d^2=1$. This implies that $G_0= \{\mathrm{id},\sigma_x,\sigma_y,\sigma_x\sigma_y\}$.

Any element of  $G$ preserves the conic bundle and preserves the two special fibres, which are the fibres of $(0:1)$ and $(1:0)$. The action on the basis is thus given by $x\mapsto \alpha x$ for some $\alpha \in \C^{*}$. Multiplying the automorphism by
$$(x,y,z)\mapsto (\lambda^2 x, \lambda^{n-2}y,\lambda^{n-1}z)$$
where $\lambda^2=\alpha^{-1}$, we get an element of $G_0$, so $$G=\{(x,y,z)\mapsto (\lambda^2 x,\pm \lambda^{n-2}y,\pm \lambda^{n-1}z)\ |\ \lambda\in \C^{*}\}\cong \C^{*}\times \Z{2}.$$

Since $F=F_{+}\cup F_{-}$ is a fibre, and any automorphism preserves the boundary $F_{+}\cup F_{-}\cup D_{n}$, the group $G$ is the group of automorphisms that preserve $F$, or equivalently that preserve $D_n$. 

The self-intersections of the curves $F_{+},F_{-},D_n$ being respectively $-1,-1,3-n$, we see that $G=\Aut(\mathcal{D}_n)$ for $n>4$. 
For $n=4$, one directly checks that $\tau\colon (x,y,z)\mapsto (-\frac{1}{2}x+\frac{\im}{2}y,\frac{3}{2}\im x-\frac{1}{2} y,z)$ is an automorphism of  order $3$ of $S_{d_n}$. Since it does not preserves the conic bundle, it extends to an automorphism which permutes cyclically the curves $F_{+},F_{-},D_{4}$ of the boundary. This implies that $\Aut(S_{d_4})=G\rtimes<\tau>$.
\end{proof}

\end{document}